\documentclass[12pt,a4paper]{amsart}
\usepackage{amsmath}
\usepackage[top=35mm, bottom=30mm, left=30mm, right=30mm]{geometry}
\usepackage{mathptmx}
\usepackage{mathrsfs}
\usepackage{verbatim}
\usepackage{url}
\usepackage[all]{xy}
\usepackage{xcolor}
\usepackage[colorlinks=true,citecolor=blue]{hyperref}
\usepackage{amsmath,amsthm}
\usepackage{amssymb}
\usepackage{enumerate}
\usepackage{graphicx}

\newtheorem{thm}{Theorem}[section]

\newtheorem{conj}[thm]{Conjecture}
\newtheorem{lem}[thm]{Lemma}

\newtheorem{prob}[thm]{Problem}

\theoremstyle{definition}

\newtheorem{exa}[thm]{Example}

\numberwithin{equation}{section}
\begin{document}


\baselineskip=20pt


\title{A new construction of counterexamples to the bounded orbit conjecture}

\author[J.~Mai]{Jiehua Mai}
\address{School of Mathematics and Quantitative
Economics, Guangxi University of Finance and Economics, Nanning,
Guangxi, 530003, P. R. China \& Institute of Mathematics, Shantou
University, Shantou, Guangdong, 515063, P. R. China}
\email{jiehuamai@163.com; jhmai@stu.edu.cn}

\author[E.~Shi]{Enhui Shi}
\thanks{*Corresponding author}
\address{School of Mathematics and Sciences, Soochow University, Suzhou, Jiangsu 215006, China}
\email{ehshi@suda.edu.cn}

\author[K.~Yan]{Kesong Yan*}
\address{School of Mathematics and Statistics, Hainan Normal
University,  Haikou, Hainan, 571158, P. R. China}
\email{ksyan@mail.ustc.edu.cn}

\author[F.~Zeng]{Fanping Zeng}
\address{School of Mathematics and Quantitative
Economics, Guangxi University of Finance and Economics, Nanning, Guangxi, 530003, P. R. China}
\email{fpzeng@gxu.edu.cn}

\begin{abstract}
The bounded orbit conjecture says that every homeomorphism on the
plane with each of its orbits being bounded must have a fixed point.
Brouwer's translation theorem asserts that the conjecture is true
for orientation preserving homeomorphisms, but Boyles'
counterexample shows that it is false for the orientation reversing
case. In this paper, we give a more comprehensible construction of
counterexamples to the conjecture. Roughly speaking,
 we construct an orientation reversing homeomorphisms $f$ on the square $J^2=[-1, 1]^2$ with
$\omega(x, f)=\{(-1, 1), (1, 1)\}$ and $\alpha(x, f)=\{(-1, -1), (1,
-1)\}$ for each $x\in (-1, 1)^2$. Then by a semi-conjugacy defined
by pushing an appropriate part of $\partial J^2$ into $(-1, 1)^2$,
$f$ induces a homeomorphism on the plane, which is a counterexample.
\end{abstract}

\keywords{fixed point, periodic point, plane homeomorphism, bounded orbit,
$\omega$-limit set, $\alpha$-limit set}

\subjclass[2010]{37E30}

\maketitle

\pagestyle{myheadings} \markboth{J. Mai, E. Shi, K. Yan and F.
Zeng}{A new construction of counterexamples to the bounded orbit
conjecture}

\section{Introduction}

In 1912, Brouwer \cite{Bro12b} obtained his translation theorem,
which has been a fundamental tool in studying surface homeomorphisms nowadays (see e.g. \cite{Ca06}).

\begin{thm}[Brouwer translation theorem]
If $f:\mathbb R^2\rightarrow \mathbb R^2$ is an orientation preserving homeomorphism and has no fixed point, then $f$
is a translation, that is each orbit of $f$ is unbounded.
\end{thm}

Some simple modern  proofs or extensions of the theorem can be found in
\cite{ Bro84,  Fa87, Fr92, Gui94, Ca05}. The following  example
shows that Brouwer translation theorem does not hold for orientation
reversing homeomorphisms (see \cite{Bo81}).

\begin{exa}\label{contex-Brouwer}
Let $h:\mathbb R^2\rightarrow \mathbb R^2, (x, y)\mapsto (-x, y)$ be the reflection
across $y$-axis. Let $g:\mathbb R^2\rightarrow \mathbb R^2$ be defined by $g(x, y)=(x, y-|x|+1)$
for $|x|<1$ and $g(x, y)=(x, y)$ for $|x|\geq 1$.
Then $g\circ h$ is an orientation reversing homeomorphism and has no fixed point, but each point $(x, y)$ with
$|x|\geq 1$ is a $2$-periodic point; specially, the orbit of $(x, y)$ is bounded.
\end{exa}

Note that the homeomorphism $g\circ h$ in Example \ref{contex-Brouwer}  has many points with an unbounded orbit,
say the points in the strip determined by $|x|<1$. This naturally leads to the following  conjecture.

\begin{conj}[Bounded orbit conjecture]
If $f:\mathbb R^2\rightarrow \mathbb R^2$ is a homeomorphism and each orbit of $f$  is bounded, then $f$
has a fixed point.
\end{conj}

In 1981,  Boyles \cite{Bo81} solved this long outstanding conjecture by constructing
a counterexample. Checking this  counterexample carefully, we found that
it is essentially a very clever improvement of Example \ref{contex-Brouwer}.
The core idea of the construction is to force each point in the strip $|x|<1$
to move upwards as well as to the left and right, so that its limit sets lie in
the vertical lines $|x|=1$. In order to achieve this construction process,
the author divided the strip $|x|<1$ into infinitely many irregular blocks, and carefully
adjusted the distances of movements of each block to ensure the convergence of relevant series.
These intricate constructions brought the readers some difficulties in understanding
why such counterexamples do exist.
\medskip

The aim of the paper is to give a more comprehensible construction of
 counterexamples. The style of our construction is more topological rather than
 geometrical, so that we can avoid the difficulties caused by irregularity of blocks and possible divergence of series
 appearing in \cite{Bo81}. In addition, the ideas of our construction are also available to
 other problems around plane homeomorphisms (see e.g. \cite{MSYZ24}).
 \medskip

In Section 3, we construct an orientation reversing homeomorphisms
$f$ on the square $J^2=[-1, 1]^2$ with $\omega(x, f)=\{(-1, 1), (1,
1)\}$ and $\alpha(x, f)=\{(-1, -1), (1, -1)\}$ for each $x\in (-1,
1)^2$ (see Lemma \ref{lem:2-2} for the details). Then by a
topological semi-conjugacy defined by pushing an appropriate part of
$\partial J^2$ into $(-1, 1)^2$, $f$ induces a homeomorphism on the
the square, and then induces a homeomorphism on the plane by a
topological conjugacy, which meets our requirements. Explicitly, we
get the following theorem.

 \begin{thm} \label{main-thm}
There exists an orientation reversing fixed point free homeomorphism $h: \mathbb{R}^2
\rightarrow \mathbb{R}^2$ such that:

 \begin{enumerate}
\item The set of all periodic points of $h$ consists of $2$-periodic points and is equal to
the subset $(-\infty, -1]\cup [1, \infty)$ of $x$-axis.

\item For each non-periodic point $x$,  the orbit of $x$ is
bounded and the limit sets $\omega(x, h)=\alpha(x, h)=\{(-1, 0), (1, 0)\}$.
\end{enumerate}
\end{thm}

Noting that both the homeomorphism given in \cite{Bo81} and that
given in Theorem \ref{main-thm}  have infinitely many periodic points, the following problem is left.

\begin{prob} \label{prob:3-3}
Does there exist  an  orientation  reversing fixed point free homeomorphism on the
plane which has no unbounded orbit and
has only finitely many periodic orbits?
\end{prob}

\section{Definitions and notations}

Let $\mathbb Z$, $\mathbb Z_+$, $\mathbb Z_-$, and $\mathbb N$ be
the sets of integers, nonnegative integers, nonpositive integers,
and positive integers, respectively. For any
\;\!$n\in\mathbb{N}$\;\!, \,write \;\!$\mathbb{N}_n=\{1,\cdots ,
n\}$. For\ \ $r\,,\,s\,\in\,\mathbb R$\,,\ \,we use \
$(r,s)$ \ to denote a point in \,$\mathbb R^{\,2}$\,.\ \ If \
$r\,<\,s$\;,\ \ we also use \ $(r,s)$\ \ to denote an open interval
in\ \,$\mathbb R$.\,\,\,These will not lead to confusion\,.\,\,\,For
example\,, \,if we write \ $(r,s)\,\in\,X$\,, \ then \ $(r,s)$ \
will be a point\,; \ if we write \ $t\,\in\,(r,s)$\,,\ \ then \
$(r,s)$ \ will be a set\,,\ \,and hence is an open interval\,.

\medskip

Let $X$ be a  metric space and $f:X\rightarrow X$ be a
homeomorphism. Let $f^{\;\!0}$ be the identity map of $X$, and let
$f^{\;\!n}=f\circ f^{\;\!n-1}$ be the composition map of $f$ and
$f^{\;\!n-1}$. For $x \in X$, the sets $O(x,
f)\equiv\{f^n(x):n\in\mathbb Z\}$, $O_+(x,
f)\equiv\{f^n(x):n\in\mathbb Z_+\}$, and $O_-(x,
f)\equiv\{f^{-n}(x):n\in\mathbb Z_-\}$ are called the {\it orbit},
{\it positive orbit}, and {\it negative orbit} of $x$, respectively.
A point $x \in X$ is called a {\it fixed point} \,of $f$ if
\;\!$f(x)=x$\;\!. We call\,$x\in X$  a {\it periodic
point}\, of $f$ if $f^{\,n}(x)=x$ \,for some $n\in\mathbb{N}$\;\!;
the smallest such $n$ is called the {\it period}\, of \,$x$ under
$f$. A periodic point of period $n$ is also called an {\it
$n$-periodic point}. Denote by $\mathrm{P}(f)$ the set of periodic
points of $f$. For $x, y\in X$, if there is a sequence of positive
integers $n_1<n_2<\cdots$ such that $f^{n_i}(x)\rightarrow y$  then
we call $y$ an {\it{{$\omega$-limit point}}} of $x$. We denote by
$\omega(x, f)$ the set of all $\omega$-limit points of $x$ and call
it the {\it $\omega$-limit set} of $x$. If $y \in \omega(x,
f^{-1})$, then we call $y$ an {\it $\alpha$-limit point} of $x$. We
denote by $\alpha(x, f)$ the set of all $\alpha$-limit points of $x$
and call it the {\it $\alpha$-limit set} of $x$.
\medskip

Let \ $X$ \ and \ $Y$ \ be topological spaces\,, \,and \
$\beta:X\,\rightarrow\,X$ \ and \ $\gamma:Y\,\rightarrow\,Y$ \ be
continuous maps\,.\ \,If there exists a continuous surjection
(\,resp. a homeomorphism) \ $\eta:X\,\rightarrow\,Y$\ \ such that \
$\eta\beta\;=\;\gamma\eta$\ , \ then\ \ $\beta$\,\,\ and\ \
$\gamma$\ \ are said to be\ \,{\it topologically semi-conjugate}
(\,resp.\ \,{\it topologically conjugate})\ ,\ \,and\ \ $\eta$\ \ is
called a\ \,{\it topological semi-conjugacy} (\;resp. a {\it
topological conjugacy})\,\ from \ $\beta$\ \ to \ $\gamma$ .

\medskip

The following lemma is well known\,(see e.g. \cite[Lemma 3.2]{MSYZ24}).

\begin{lem} \label{lem:2-1} Let\ \ $\beta:X\,\rightarrow\,X$\ \ and \ $\gamma:Y\,\rightarrow\,Y$ \
with a topological semi-conjugacy \ $\eta:X\,\rightarrow\,Y$\ \ be
as above\;. \ If \;both\ \ \,$X$\ \,and\ \ \,$Y$\ \,are compact
metric spaces\;, \,then \

\vspace{1mm}$(1)$\ \ \ For \,any \ $x\;\in\;X$\,, \ the \
$\omega$-limit set \,\ $\omega\big(
\eta(x)\,,\,\gamma\big)\,=\,\;\eta\big(\omega(x\,,\,\beta)\big)\ ;$

\vspace{1mm}$(2)$ \ \ If \;both \ $\beta$ \ and \ $\gamma$ \,are
\,homeomorphisms\,, \ then , \ for \,any \ $x\;\in\;X$\,, \ the
\,$\alpha$-limit set\ \ \,$\alpha\big(\eta(x)\,,\,\gamma\big)
\,=\,\;\eta\big(\alpha(x\,,\,\beta)\big)$\ .
\end{lem}

We always write\ \ $J=\,[\,-1\,,\,1\,]$\;.\ \ \,
For any compact connected manifold $M$,\ \ denote by \,$\partial M$
\,the boundary\,, \,and by $\stackrel{\ \circ}{M}$ \,the interior of
\,$M$. \ Specially, \ we have \ $\partial J\,
=\,\{\,-1\,,\,1\:\!\}$\,,  \ $\stackrel{\
\circ}{J}\,=\,(\,-1\,,1\:\!)$\,, $\partial J^{\,2}=\{-1, 1\}\times J\ \cup\ J\times \{-1, 1\}$, \ and
\ $\stackrel{\circ}{J}$$^2=(-1, 1)^2$.

\medskip

Define the
homeomorphism \ $f_{01}\,:\,J\,\to\;J$\ \ by
$$
f_{01}(s)\,=\;
   \left\{\begin{array}{ll}(s+1)/2\,, & \mbox{\ \ \ \ \ if \ }\;0\,\le\,s\,\le1\;; \\
                        \!\;\;s+1/2\,, & \mbox{\ \ \ \ \ if \ \ }-1/2\,\le\,s\le\,0\;;  \\
                             \;2s+1\,, & \mbox{\ \ \ \ \ if \ }-1\,\le\,s\le\,-1/2\ .
\end{array}\right.
$$
Then we define the {\it standard vertical shift} \ $f_{02}\,:\,J^{\,2}\,\to\;J^{\,2}$  by\
, \ \,for any \ $(r,s)\,\in\,J^{\,2}$\ ,
$$
f_{02}(r,s)\ =\ \big(\,r\,,\,f_{01}(s)\,\big)\ .
$$
For any\ \ $s\,\in\,J$\,,\ \ \,write \ $J_s\,= \,J\times\{s\}$\;. \
A homeomorphism\ $\zeta:J^{\,2}\to\,J^{\,2}$
\ is said to be \,{\it normally rising} \ if
$$ \zeta(J_s)\;=\,f_{02}(J_s),
\hspace{5mm} \mbox{for \ any}\ \ \,s\,\in\,J.$$
Define the {\it level reflection} $\Psi: J^2 \rightarrow J^2$ by
$$\Psi(r, s)=(-r,s), \hspace{10mm} \mbox{for any\ } (r,s) \in J^2.$$
Then $\Psi^2$ is the identity on $J^2$, and for any $(r,s)\in J^2$, we have
$$\Psi f_{02}(r,s)=f_{02}\Psi(r,s)=(-r, f_{01}(s)),
$$
that is $\Psi f_{02}=f_{02}\Psi$.

\section{Normally rising orientation reversing homeomorphisms on $J^2$}

Write
$$v_1=(-1,1), \hspace{8mm} v_2=(1,1), \hspace{8mm} v_3=(-1,-1),
\hspace{8mm} v_4=(1,-1).$$ Then $v_1, v_2, v_3$ and $v_4$ are the
four vertices of the square $J^2$.

\begin{lem} \label{lem:2-2}
There exists a normally rising orientation reversing homeomorphism
$f: J^2 \rightarrow J^2$ such that $\omega(x, f)=\{v_1, v_2\}$ and
$\alpha(x, f)=\{v_3, v_4\}$ for any $x \in \stackrel{\
\circ}{J}$$^2$, and $f|\partial J^2=\Psi f_{02}|\partial J^2$.
\end{lem}

\begin{proof}
Take $a_1=1/4>a_2>a_3> \cdots >0$ such that  $a_n \rightarrow 0$ as
$n \rightarrow \infty$. Let
$$D_0 = J \times [0, 1/2]\;\,,\ \ \ \ B_{0n} = J \times [a_n , 1/2]\;\,,\ \ \ \
F_{0n} = J \times [0 , a_n]\;\,.$$ Then\, $D_0 = B_{0n} \cup F_{0n}$
for any\, $n\in \mathbb{N}$\ ,\ and\ $B_{01} \subset B_{02} \subset
B_{03}\subset \cdots$\ and\ $F_{01} \supset F_{02} \supset F_{03} \supset \cdots$.
For any\ \,$i \in \mathbb{Z}$\ ,\ \,let
$$D_i = f_{02}^i(D_0)\;\,,\ \ \
\ B_{in} = f_{02}^i(B_{0n})\;\,,\ \ \ \ F_{in} =
f_{02}^i(F_{0n})\;\,.$$ Then\ \, $D_{i-1} \cap D_i = f_{02}^i(J_0)$
,\ \ and \ \,$D_i = B_{in} \cup F_{in}$\ \,for any\ \,$n\in
\mathbb{N}$\ \ and any\ \,$i \in \mathbb{Z}$\ .\ \ Write
$$R_0 = J \times [0,1]\ ,\ \ \ \ \ R_1 = J \times [1/2, 1]\ .$$
Then\ \ $R_0 = D_0 \cup R_1$\ ,\ \ $R_1 = \bigcup \{D_i : i \in
\mathbb{N}\} \cup J_1$\ , \ \ and\ \ $\Psi f_{02}(R_0) = f_{02}(R_0)
= R_1$\ .

\medskip

Take\ \,$b_1 = 1/2 < b_2 < b_3 < \cdots < 1$\ \,such that\ \, $b_n
\rightarrow 1$\ \,as\ \, $n\rightarrow \infty$\ .\ \ For any\ \,$n
\in \mathbb{N}$ ,\ \,take a homeomorphism\ \, $\varphi_n : J
\rightarrow J$\ \,such that\ \,
$$0 < \varphi_n(r) - r \leq 2b_n /n \ \ \ \
\mbox{for any}\ \ r \in\ \stackrel{\ \circ}{J}\,$$ and
$$\varphi_n(r) - r = 2b_n /n\ \ \ \ \mbox{for any}\ \ r \in [-b_n ,\ \,b_n-2b_n
/n].$$Then\ \, $\varphi_n^n(-b_n) = b_n$\ \ and\ \ $\varphi_n^n([-
b_n , 1]) = [b_n , 1]$.

\medskip

For any\ \,$n \in \mathbb{N}$\ ,\ \,write\ \,$k(n) = n(n+1)$\ .\
\,Then\ \,$k(n)$\ \,is an even number ,\ \,$k(1) = 2$\ ,\ \,$k(2) =
6$\ ,\ \,$k(3) = 12$\ ,\ \, $\cdots$ \ \,,\ \,and\ \,$k(n +1) = k(n)
+2n +2$\ .\ \ \ Let $$\mathbb{N}(n)=\{k(n)\ ,\;\,k(n)+1\ ,\;\,
\cdots \ ,\;\,k(n+1)-1\}.$$ Then\ \,$\mathbb{N}(n)$\ \,is a subset
of\ \,$\mathbb{N}$\ \,containing\ \,$2n + 2$\ \,elements ,\ \,and\
\,$\{ \mathbb{N}(n) : n\in \mathbb{N}\}$\ \,is a partition of\;\,
$\mathbb{N}-\{1\}$ .

\medskip

For any given\ \,$i \in \mathbb{N}- \{1\}$\ ,\ \, let
$$n(i)=\max\{m \in \mathbb{N} : k(m) \leq i \}.$$
Then\, $n(i)$\ \,is the number in\ \,$\mathbb{N}$\ \,such that\ \,$i
\in \mathbb{N}(n(i))$ ,\ \,and we have\ \,$n(i)=1$\ \,for\ \,$i =2,
3 , 4 , 5$ ,\ \,$n(i)=2$ \,for\ \,$i=6 , \cdots, 11$ ,\ \ \,$n(i) =
3$\ \,for\ \,$i = 12 , \cdots , 19$ ,\ \ \ and so on .

\medskip

Clearly, we can uniquely define a homeomorphism\ \, $\Phi: R_1
\rightarrow R_1$\ \,by the following conditions :

\medskip

$(\Phi.1)$\ \  $\Phi|(D_1 \cup \partial R_1)=id$\ ,\ \ and\ \
$\Phi(J_s) = J_s$\ \ for any\ \ $s \in [1/2 , 1]$\ .

\medskip

$(\Phi.2)$\ \ For any\ \ $i \in \mathbb{N} - \{1\}$\ \ and any\ \
$(r, s) \in B_{i n(i)}$\ ,\ \ if\;\,$i$\;\,is odd then\ \ $\Phi(r ,
s) = (r , s)$\ ;\ \ if\;\,$i$\;\,is even then\ \ $\Phi(r , s)
=(\varphi_{n(i)}(r) , s)$.

\medskip

$(\Phi.3)$\ \ For any\ \ $i \in \mathbb{N} - \{1\}$\ \ and any\ \ $r
\in J$\ ,\ \ $\Phi|((\{r\}\times J) \cap F_{in(i)} )$\ \ is linear .

\medskip

Note that in the above definition of\ \, $\Phi : R_1 \rightarrow
R_1$\ , \ \,we divide\;\,$R_1$\;\,in the manner shown by the
following equality :\ \
\begin{align*}
R_1 &=\, D_1 \cup D_2 \cup \cdots \cup D_i \cup D_{i+1} \cup \cdots \cup J_1\\
&= \, D_1 \cup F_{2 n(2)} \cup B_{2 n(2)} \cup \cdots \cup F_{i
n(i)} \cup B_{i n(i)} \cup F_{i+1, n(i+1)} \cup B_{i+1, n(i+1)} \cup
\cdots \cup J_1\ ,
\end{align*}
and the definition of\;\, $\Phi|B_{in(i)}$\;\,is directly given by
means of\;\, $\varphi_{n(i)}$\;\,if\;\,$i$\;\,is even ,\ \ and the
definition of\;\, $\Phi|F_{in(i)}$\;\,is given by means of\;\,
$\Phi|B_{i-1,n(i-1)}$\;\,and\;\, $\Phi|B_{in(i)}$\ .

\medskip

Let\ \ $\eta = \Phi \Psi f_{02} | R_0 : R_0 \rightarrow  R_1$ .\ \
Then\;\,$\eta$\;\,is a homeomorphism ,\ \ $\eta | (D_0 \cup \partial R_0)
=\Psi f_{02} | (D_0 \cup \partial R_0)$\ ,\ \ and\ \ $\eta(J_s) = f_{02}
(J_s)$\ \ for any\ \,$s \in [0,1]$\ .

\medskip

{\bf Claim 1}.\ \ {\it For any given\ \ $x_0=(r_0, s_0) \in \;
\stackrel{\ \circ}{J}\, \times\, (0, 1/2]$, \ \ $\lim_{i \rightarrow
\infty}\eta^{2i}(x_0) = v_2$\ .}

\medskip

{\bf Proof of Claim 1.}\ \ For any\ \,$i \in \mathbb{N}$\ ,\ \,let\
\, $x_i = (r_i , s_i) = h^i(x_0)$\ .\ \ Then\ \,$r_i \in\;
\stackrel{\ \circ}{J}$ \ \,and\ \,$s_i = f_{01}^i(s_0) \rightarrow 1$\
\,as\ \,$i \rightarrow \infty$.\ \ Take a\ \, $\mu \in \mathbb{N}$\
\,such that\ \, $x_0 \in B_{0\mu}$ \ .\ \ Then for any\ \, $i \geq
k(\mu)$,\ \,we have\ \ $n(i) \geq \mu$  \ \ and\ \ $x_i \in B_{i\mu}
\subset B_{in(i)}$\ .\ \,Hence,\ \,for any odd\ \,$i \geq k(\mu)$ ,\
\ we have
$$\eta(x_i)=\Phi \Psi f_{02}(r_i,s_i)=\Phi(-r_i,\ s_{i+1})=
\left(\varphi_{n(i+1)}(-r_i)\,\;\,s_{i+1}\right)=(r_{i+1},
s_{i+1})$$ with\ \ $r_{i+1}=\varphi_{n(i+1)}(-r_i)>-r_i$\ ,\ \ \ and
for any even\ \,$i\geq k(\mu)$ , \ \ we have
$$\eta(x_i)=\Phi \Psi f_{02}(r_i, s_i)=\Phi(-r_i, s_{i+1})=(-r_i, s_{i+1})$$
with\ \ $r_{i+1}=-r_i$\ , \ \ and
$$\eta^2(x_i)=\Phi \Psi f_{02}(-r_i, s_{i+1})=\Phi(r_i, s_{i+2})=(r_{i+2}, s_{i+2})$$
with\ \
\begin{equation} \label{eq:3-1}
r_{i+2}\ \,=\ \, \varphi_{n(i+2)}(r_i)\ \,>\ \, r_i\ .
\end{equation}
Therefore,\ \ for any even\ \,$i \geq k(\mu)$,\ \ we have
\begin{equation} \label{eq:3-2}
-1 < r_i < r_{i+2} < r_{i+4} < r_{i+6} <\;\,\cdots \;\,< 1\;\,.
\end{equation}

In order to prove\ \ $\lim_{i \rightarrow \infty}r_{2i}=1$\ ,\ \
consider any given\ \, $\varepsilon>0$.\ \,Take an\ \,$m=m(\varepsilon) \in
\mathbb{N} - \mathbb{N}(\mu)$ \ \,such that\ \,$b_m>1-\varepsilon$\
\,and\ \,$-b_m < r_{k(\mu)}$\;\,.\ \ Then\ \,$r_{k(m)} > r_{k(\mu)}
> -b_m$\ .\ \ Consider the point\ \ $y_{k(m)}=(-b_m, s_{k(m)}) \in B_{k(m),
m}$ \ .\ \ Let\ \ $\beta_{k(m)}=-b_m$\;\,.\ \ For any\ \,$i\in
\mathbb{N}$\ ,\ \ let
$$y_{k(m)+ i}\;\,=\;\,(\beta_{k(m)+ i}\ ,\;\,s_{k(m)+i})\;\,=\;\,\eta^i(y_{k(m)})\;\,.$$
Similar to \eqref{eq:3-1} ,\ \ for any even\ \,$i \geq k(m)$,\ \ we
have\ \
\begin{equation} \label{eq:3-3}
  \beta_{i+2}\ \,=\ \, \varphi_{n(i+2)}(\beta_i)\ \,>\ \, \beta_i\ .
\end{equation}
Note that\ \ $n(i)=m$\ \ and\ \ $\varphi_{n(i)}=\varphi_m$ \ \ for\
\ $k(m)\leq i < k(m+1)=k(m)+2m+2$\ ,\ \ and\ \
\begin{equation} \label{eq:3-4}
 \varphi_m(r) = r + 2b_m /m\ \ \ \ \mbox{for any}\ \ r \in [-b_m,\ \,b_m - 2b_m/m]\ .
\end{equation}
In \eqref{eq:3-3} and \eqref{eq:3-4}\ ,\ \ taking\ \ $i=k(m)$\ ,\ \
$k(m)+2$\ ,\ \, $\cdots$ \ ,\ \,$k(m)+2n-2$\ \ and\ \
$r=\beta_{k(m)}$\ ,\ \, $\beta_{k(m)+2}$\ ,\ \, $\cdots$ \ ,\ \,
$\beta_{k(m)+2n-2}$\ ,\ \ we get
\begin{align*}
\beta_{k(m)+2}\; &= \varphi_{n(k(m)+2)}(\beta_{k(m)})=\beta_{k(m)} + 2b_m /m \ \,,\\
\beta_{k(m)+4}\; &= \varphi_{n(k(m)+4)}(\beta_{k(m)+2}) =
\beta_{k(m)+2} + 2b_m/m =\beta_{k(m)}+ 4b_m/m\ \,,\\
& \cdots \cdots\\
\beta_{k(m)+ 2m}&= \varphi_{n(k(m)+2m)}(\beta_{k(m)+2m-2})
=\beta_{k(m)+2m-2} + 2b_m /m \\
&=\beta_{k(m)} + 2mb_m /m = - b_m + 2b_m = b_m > 1-\varepsilon\ \,.
\end{align*}

Define a projection\ \ $p: J^2 \rightarrow J$\ \ by\ \ $p(x) = r$\ \
for any\ \ $x = (r, s)\in J^2$ ,\ \,that is ,\ \ $p(x)$\ \,is the
abscissa of\ \,$x$\ .\ \ Clearly , \ \,for any\ \ $s \in [0,1]$\
\,and any\ \ $x, y \in J_s$ ,\ \,if\ \ $p(x)> p(y)$\ \ then\ \
$p\eta^2(x)> p\eta^2(y)$.\ \ Thus  we have
$$r_{k(m)+2m} > \beta_{k(m)+2m} > 1-\varepsilon$$
since\ \ $r_{k(m)}>-b_m=\beta_{k(m)}$,\ \ and hence ,\ \ by
\eqref{eq:3-2},\ \ we have
$$\varepsilon > 1-r_{k(m)+ 2m} > 1 - r_{k(m)+2m + 2} >
1 - r_{k(m)+ 2m + 4} > \cdots > 0.$$ This means that\ \ $\lim_{i
\rightarrow \infty}r_{2i}=1$,\ \,which with\ \ $\lim_{i \rightarrow
\infty}s_i = 1$\ \,implies\ \ $\lim_{i \rightarrow \infty}
\eta^{2i}(x_0)=\lim_{i \rightarrow \infty}(r_{2i}, s_{2i})=(1,1)=v_2$.\
\ \ Claim 1 is proved.

\medskip

Since\ \ $\eta(v_2) = v_1$,\ \ from Claim 1 we get

\medskip

{\bf Claim 2.}\ \ {\it For any given\ \ $x_0=(r_0, s_0) \in\;
\stackrel{\ \circ}{J} \times (0, 1/2]$, \ \ $\lim_{i \rightarrow
\infty} \eta^{2i+1}(x_0) = v_1$\ ,\ \ and\ \ $\omega(x_0, \eta)=\{v_1 ,
v_2\}$.}

\medskip

Define the {\it vertical reflection}\ \, $\Psi_v: J^2 \rightarrow
J^2$\ \,by $$\Psi_v(r, s)=(r, -s) \ \ \ \ \ \mbox{for any}\ \ (r, s) \in
J^2.$$ Then\ \ $\Psi_v^2 = id$\ ,\ \ $\Psi_v
f_{02}=f_{02}^{-1}\Psi_v$\ ,\ \ and\ \ $\Psi_v(D_i)=D_{-i-1}$\ \ for
any\ \ $i\in \mathbb{Z}$\ .\ \ Write
$$R_{-1}=J \times [-1, 0]\ ,\ \ \ \ \ R_{-2}=J \times [-1, -1/2]\ .$$
Then\ \ $J^2=R_0 \cup R_{-1}$\ , \ \ $R_{-1}=\Psi_v(R_0)=D_{-1} \cup
R_{-2}$\ ,\ \ and\ \ $R_{-2}=\Psi_v(R_1)=\bigcup\{D_{-i-1} : i \in
\mathbb{N}\} \cup J_{-1}$\ .\ \ Let
$$\zeta\;\,=\;\, \Psi_v\, \eta\, \Psi_v\,|\,R_{-1}\;\,:\;\,R_{-1}\;\, \rightarrow \;\,R_{-2}\;\,.$$
Then\ \,$\zeta$\ \,is a homeomorphism. \ \ From Claims 1 and 2 we get

\medskip

{\bf Claim 3.}\ \ {\it For any\ \ $x \in\; \stackrel{\ \circ}{J}\,
\times \, [-1/2, 0)$,\ \ one has\ \ $\lim_{i \rightarrow
\infty}\zeta^{2i}(x) = v_4$\ ,\ \ \ $\lim_{i \rightarrow \infty}
\zeta^{2i+1}(x) = v_3$\ ,\ \ \ and\ \ $\omega(x, \zeta)=\{v_3, v_4\}$.}

\medskip

We now define\ \ $f : J^2 \to J^2$\ \ by\ \
$$f\,|\, R_0 = \eta\ ,\ \ \ \ f\,|\, D_{-1} =\Psi\,f_{02}\,|\,D_{-1}\ ,
\ \ \ \ \mbox{and}\ \ \ f \,|\, R_{-2} = \zeta^{-1}\ .$$ Then\ \,$f$\
\,is a normally rising orientation reversing homeomorphism\ ,\ \
and\ \,$f | \partial J^2 = \Psi \,f_{02}\,|\, \partial J^2$.\ \
Further ,\ \,for any\ \, $x \in \stackrel{\ \circ}{J}$$^{\,2}$,\ \,there
exists\ \, $i$\ \,and\ \,$j \in \mathbb{Z}$\ \ such that\ \, $f^i(x)
\in \; \stackrel{\ \circ}{J}\, \times \, (0, 1/2]$\ \,and\ \, $f^j(x)
\in \;  \stackrel{\ \circ}{J}\, \times [-1/2, 0)$ , \ \ and hence we
always have
$$\omega(x, f)=\omega(f^i(x)\ ,\;\,f)= \omega(f^i(x)\ ,\;\,\eta)=\{v_1, v_2\}\;\,,$$
and
$$\alpha(x, f)=\alpha(f^j(x)\ ,\;\,f)=\omega(f^i(x)\ ,\;\,f^{-1})=\omega(f^i(x)\ ,\;\,\zeta) = \{v_3, v_4\}\;\,.$$
Lemma 3.1 is proved.
\end{proof}

\section{Homeomorphisms of the plane which have neither fixed point nor unbounded orbit}

In this section, we will use Lemma \ref{lem:2-2} to construct an
orientation reversing homeomorphism on the plane which has neither
fixed point nor unbounded orbit.

 \begin{thm} \label{main-thm'}
There exists an orientation reversing fixed point free homeomorphism $h: \mathbb{R}^2
\rightarrow \mathbb{R}^2$ such that:

 \begin{enumerate}
\item The set of all periodic points of $h$ consists of $2$-periodic points and is equal to
the subset $(-\infty, -1]\cup [1, \infty)$ of $x$-axis.

\item For each non-periodic point $x$,  the orbit of $x$ is
bounded and the limit sets $\omega(x, h)=\alpha(x, h)=\{(-1, 0), (1, 0)\}$.
\end{enumerate}
\end{thm}

\begin{proof}
Let the four vertices\ \ $v_1\,=\,(-1,\, 1)\;,\ \
v_2\,=\,(1,\,1)\;,\ \ v_3\,=\,(-1,\,-1)$\ \ and\ \
$v_4\,=\,(1,\,-1)$ of\ \ $J^{\,2}$\ \ be as in Lemma \ref{lem:2-2}\
.\ \ Write\ \ $v_5\,=\,(-1,\,0)\;,\ \  v_6\,=\,(1,\, 0)\;,\ \
v_7\,=\,(0,\,1)\;,\ \ v_8\,=\,(0,\,-1)\;,\ \ v_9\,=\,(-1/2,\, 0)\;,\
\ v_0\,= \,(1/2,\, 0)$\ . \ Then \ $v_5\;,\ \,v_6\;,\ \,v_7\;,\
\,v_8$ \ are the four midpoints of the four sides of\
\,$J^{\,2}$\,,\ \ respectively\,.\ \ Let the level reflection\ \ $\Psi$
\ and the vertical reflection \ $\Psi_v$ \ be the same as in the proof
of Lemma \ref{lem:2-2} .\ \ Clearly\,,\,\,\,we can take a continuous
map \,\,\ $\xi:J^{\,2}\,\to\,J^{\,2}$\,\,\ satisfying the following
conditions\ :

\medskip

(1)\ \ \ $\xi(0,s)\;=\;(0,s)$\ \ \,for any\ \ $s\,\in\,J$\;,\ \ \
and \ \ $\xi(r,0)\;=\;(\,r/2\;,\;\!0)$\ \ \,for any\ \
$r\,\in\,J$\;;

\medskip

(2)\ \ \ $\xi\big(\,[v_6,v_2]\,\big)\;=\;\{v_0\}$\ ,\ \ \
$\xi\big(\,[v_7
,v_2]\,\big)\;=\;[v_7,v_2]\,\cup\,[v_2,v_6]\,\cup\,[v_6,v_0]$\ ,\ \
\ and \ \ $\xi\;| \;[v_7,v_2]$\ \ is an injection\;;

\vspace{2mm}(3)\ \ \ $\xi\Psi\;=\;\Psi\xi$\ ,\ \ \ and\ \ \
\,$\xi\Psi_v\;=\;\Psi_v\xi$\ ;

\vspace{2mm}(4)\ \ \ $\xi\;|\;\stackrel{\ \circ}{J}$$^{\,2}$\ \ is
an injection\;, \ \ and\ \ \ \,$\xi(\stackrel{\
\circ}{J}$$^{\,2})\;=\;\stackrel{\
\circ}{J}$$^{\,2}\,-\,[v_5,v_9]\,-\,[v_ 0,v_6]$\ .

\noindent Let the homeomorphism \ $f:J^{\,2}\to J^{\,2}$\ \ be as in
Lemma \ref{lem:2-2}\ .\ \ Define a map\ \ $g:J^{\,2}\to J^{\,2}$\ \
by\ \ $g\;=\;\xi\,f\,\xi^{\;-1}$\,.\ \ Note that\,,\ \,if\ \
\,$x\,\in\;[\,v_5\,,v_9\,]\,\cup\;[\,v_0\,,v_6\,]$\ ,\ \ then \
$\xi^{\;-1}(x)$\ \ contains more than one point\,,\ \,but\ \
$\xi\,f\,\xi^{\;-1}(x)$\ \ still contains only one point\,.\ \,Thus
\ $g$ \ is well defined\,. \ It is easy to check that \ $g$ \ is a
bijection\,,\ \,and is continuous\,. \,Thus \ $g:J^{\,2}\to J
^{\,2}$\ \ is a homeomorphism\,.\ \ Moreover\,, \ from the
definition \ $g\;=\;\xi\,f \,\xi^{\;-1}$ \ we see that\ \ $g$\ \ has
the following properties\;:

\medskip

(P.A)\ \ \ $g$ \ is orientation reversing\ ;

\medskip

(P.B)\ \ \ $f$\ \ and\ \ $g$ \ are topologically semi-conjugate\,,\
\,and \ $\xi$ \ is a topological semi-conjugacy from\ \ $f$\ \ to \
$g$\ ;

\medskip

(P.C)\ \ \ The set of periodic points of \ $g$ \ is \ $P(g)\;=\
\partial J^{\, 2}\,\cup\;[\,v_5\,,v_9\,]\,\cup\;[\,v_0\,,v_6\,]\;=\
\xi(\partial J^{\,2})$\;, \ with \ $g \:|\,P(g)\;=\ \Psi\:|\,P(g)$ .
\ Hence\,, \,in \ $P(g)$\ , \ only \ $v_7$ \ and \ $v_8$ \ are fixed
points , \,and other points of \ $P(g)$ \ are \,2-periodic points\;;

\vspace{2mm}(P.D)\ \ \ For any\ \ $x\,\in\,J^{\,2}-P(g)\;=\
\xi(\stackrel{\ \circ}{J}$$^{\, 2})$\ , \ from \,Lemma \ref{lem:2-2}
\,and \,Lemma \ref{lem:2-1} \, we get
$$\omega(x,g)\ =\ \xi\big(\,\omega(\xi^{\,-1}(x)\,,\,f)\,\big)\ =\
\xi\big(\,\{v_1,v_2\}\,\big)\ =\ \{v_9,v_0\}$$ and
$$\alpha(x,g)\ =\ \xi\big(\,\alpha(\xi^{\,-1}(x)\,,\,f)\,\big)\ =\
\xi\big(\,\{v_3,v_4\}\,\big)\ =\ \{v_9,v_0\};$$

(P.E)\ \ \ By \,(P.D) \,it is easy to see that\,, \,for any\ \
$x\,\in\, J^{\,2}-P(g)$\ ,\ \ there exists an \
$\varepsilon\,=\;\varepsilon(x)\,\in\,(\,0,1/4\,]$\ \ such that \
$O(x,g)\,\subset\;[\,-1+\varepsilon\,,
\,1-\varepsilon\,]^{\,2}$\ .

\medskip

We now define a homeomorphism\,\,\ $\psi\,:\,\stackrel{\
\circ}{J}$$^{\,2}\,\to\, \Bbb R^{\,2}$ \ by \ $\psi(r,s)\ =\
\big(\,{\rm tg}(\;\!\pi r/2)\ ,\;{\rm tg}(\;\!\pi s /2)\,\big)$ \
for any \ $(r,s)\,\in\,\stackrel{\ \circ}{J}$$^{\,2}$\,,\ \ and then
define a homeomorphism \ $h\,:\,\Bbb R^{\,2}\,\to\,\Bbb R^{\,2}$ \
by \ $h\;=\ \psi\,g\,\psi^{\,-1 }$\,. \ From \ (P.A) -- (P.E) \ we
see that\ \ $h$\ \ has the following properties\;:

\medskip

(P.F)\ \ \ $h$ \ is orientation reversing\ ;

\medskip

(P.G)\ \ \ $h$ \ and \ $g|(-1, 1)^2$ \ are topologically conjugate\,, \,with \
$\psi$ \ being a topological conjugacy from\ \ $g$\ \ to \ $h$\ ;

\medskip

(P.H)\ \ \ The set of periodic points of \ $h$ \ is
$$\mbox{$P(h)\;=\ \psi\big(\,P(g)\,\cap\,\stackrel{\ \circ}{J}$$^{\,2}\,\big)\;=\ (\,-
\infty,-1\,]\times\{0\}\;\cup\;[\,1,\infty)\times\{0\}\;,$}$$ with \
$h(r,0)\;=\;(-r,0)$ \ for any \ $(r,0)\,\in\,P(h)$\ .\ \ Hence \ $h$
\ has no fixed point\,,\,\,\,and all points in\,\ \,$P(h)$\,\,\ are
\,2-periodic points\ .\ \ Particularly\,, \,write \ $w_1\,=\,(-1,0)\
,\ \ w_2\,=\,(1,0)$,\ \ and let \ $O_2\; =\;\{w_1,w_2\}$ . \ Then
\ $O_2$ \ is a periodic orbit of \ $h$ .

\medskip

(P.I)\ \ \ For any\ \ $x\,\in\,\Bbb R^{\,2}-P(h)\;=\
\psi(J^{\,2}-P(g))$\ , \ by \,(P.D) \,it holds that
$$\omega(x,h)\ =\,\,\psi\big(\,\omega(\psi^{\,-1}(x)\,,\,g)\,\big)\ =\
\psi\big(\,\{v_9,v_0\}\,\big)\ =\ \{w_1,w_2\}\ =\ O_2$$ and
$$\alpha(x,h)\ =\,\,\psi\big(\,\alpha(\psi^{\,-1}(x)\,,\,g)\,\big)\ =\
\psi\big(\,\{v_9,v_0\}\,\big)\ =\ \{w_1,w_2\}\ =\ O_2\;;$$

\medskip

(P.J)\ \ \ For any \ $x\,\in\,\Bbb R^{\,2}$ , \ the orbit \ $O(x,h)$
\ is bounded .

\noindent From\ \,(P.F) -- (P.J) \,we see that\,\,\ $h$\,\,\ has the
properties mentioned in Theorem \ref{main-thm'}, \ and the proof is
complete\,.
\end{proof}
\vspace{10mm}

\subsection*{Acknowledgements}$\\$
We express our gratitude for the suggestions and comments provided by the reviewers.$\\$
Jiehua Mai and Fanping Zeng are supported by NNSF of China (Grant
No. 12261006) and Project of Guangxi First Class Disciplines of
Statistics (No. GJKY-2022-01); Enhui Shi is supported by NNSF of
China (Grant No. 12271388); Kesong Yan is supported by NNSF of China
(Grant No. 12171175).


\begin{thebibliography}{HD}

\normalsize
\baselineskip=17pt


\bibitem{Bo81} S. M. Boyles, \,{\it A counterexample to the bounded orbit conjecture},\;
Trans. Amer. Math. Soc., \,{\bf 266}\,(1981),\, 415\;--\;422.

\bibitem{Bro12b} L. E. J. Brouwer, \,{\it Beweiss des ebenen Translationssatzes},
Math. Ann., \,{\bf 72}\,(1912),\, 37\;--\;54.

\bibitem{Bro84} M. Brown, \,{\it A new proof of Brouwer's lemma on translation
 arcs}, \,Houston J. Math., \,{\bf 10}\,(1984), \,35\;--\;41.

\bibitem{Fa87} A. Fathi, \,{\it An orbit closing proof of Brouwer's
lemma on translation arcs},\, L'enseignement Mathematique, \;{\bf
33}\,(1987),\, 315\;--\;322.

\bibitem{Fr92} J. Franks, \,{\it A new proof of the Brouwer plane translation theorem},
\, Ergodic Theory Dynam. Systems, \,{\bf 12}\, (1992), \,
217\;--\;226.

\bibitem{Gui94} L. Guillou,\, {\it Th$\acute{e}$or$\grave{e}$me de translation plane de
Brouwer et g$\acute{e}$n$\acute{e}$ralisations du
th$\acute{e}$or$\grave{e}$me de Poincar$\acute{e}$-Birkhoff}, \,
Topology, \, {\bf 33}\,(1994), \, 331--351.

\bibitem{Ca06} P. Le Calvez, \,{\it From Brouwer theory to the study of homeomorphisms of surfaces}, \;
European Mathematical Society (EMS), Z\"urich, 2006, 77\;--\;98.

\bibitem{Ca05}P. Le Calvez,  \, {\it Une version feuillet\'ee \'equivariante du th\'eor\'eme de translation de Brouwer},
Inst. Hautes \'Etudes Sci. Publ. Math.,\, {\bf 102}\, (2005), 1\;--\;98.

\bibitem{MSYZ24} J. Mai, E. Shi, K. Yan and F. Zeng, {\it Can points of bounded orbits surround points of unbounded
orbits?} arXiv: 2404.13642.

\end{thebibliography}
\end{document}